\documentclass[a4paper]{amsart}
\usepackage{graphicx}
\usepackage{amssymb}
\usepackage{mathtools}
\usepackage{stmaryrd}
\usepackage[mathcal]{euscript}
\usepackage{tikz}
\usepackage{tikz-cd}
\usepackage{hyperref}
\hypersetup{%
  bookmarksnumbered=true,%
  colorlinks=true,%
  linkcolor=blue,%
  citecolor=blue,%
  filecolor=blue,%
  menucolor=blue,%
  urlcolor=blue,%
  bookmarksopen=true,%
  bookmarksdepth=2,%
  pageanchor=true}

\title{On Matlis reflexive modules}

\author{Henning Krause}
\address{Fakult\"at f\"ur Mathematik\\
Universit\"at Bielefeld\\ D-33501 Bielefeld\\ Germany}
\email{hkrause@math.uni-bielefeld.de}

\theoremstyle{plain}
\newtheorem{thm}{Theorem}[section]
\newtheorem{prop}[thm]{Proposition}
\newtheorem{lem}[thm]{Lemma} 

\newtheorem{cor}[thm]{Corollary}

\theoremstyle{definition}

\newtheorem{exm}[thm]{Example}

\theoremstyle{remark}
\newtheorem{rem}[thm]{Remark}

\numberwithin{equation}{thm}

\newcommand{\add}{\operatorname{add}}

\newcommand{\arno}{\operatorname{arno}}
\newcommand{\art}{\operatorname{art}}

\newcommand{\Coker}{\operatorname{Coker}}

\newcommand{\End}{\operatorname{End}}

\newcommand{\fl}{\operatorname{fl}}

\newcommand{\Hom}{\operatorname{Hom}}

\renewcommand{\Im}{\operatorname{Im}}

\newcommand{\Ker}{\operatorname{Ker}}

\renewcommand{\mod}{\operatorname{mod}}

\newcommand{\noeth}{\operatorname{noeth}}

\newcommand{\refl}{\operatorname{refl}}

\newcommand{\soc}{\operatorname{soc}}

\newcommand{\Spec}{\operatorname{Spec}}

\newcommand{\iso}{\xrightarrow{\raisebox{-.4ex}[0ex][0ex]{$\scriptstyle{\sim}$}}}

\newcommand{\longiso}{\xrightarrow{\ \raisebox{-.4ex}[0ex][0ex]{$\scriptstyle{\sim}$}\ }}

\newcommand{\lto}{\longrightarrow}
\newcommand{\smatrix}[1]{\left[\begin{smallmatrix}#1\end{smallmatrix}\right]}

\newcommand{\xto}{\xrightarrow}
\newcommand*{\intref}[2]{\def\tmp{#1}\ifx\tmp\empty\hyperref[#2]{\ref*{#2}}\else\hyperref[#2]{#1~\ref*{#2}}\fi}

\def\A{\mathcal A} 
 
\def\C{\mathcal C}
\def\D{\mathcal D} 
 
\def\F{\mathcal F}

\def\T{\mathcal T}

\def\bbN{\mathbb N}

\def\bbZ{\mathbb Z}

\newcommand{\fra}{\mathfrak{a}}

\newcommand{\frm}{\mathfrak{m}} 
\newcommand{\frn}{\mathfrak{n}}

\def\a{\alpha}

\def\p{\phi}

\def\La{\Lambda}

\def\Om{\Omega}

\begin{document}



\begin{abstract}
  Matlis duality for modules over commutative rings gives rise to the
  notion of Matlis reflexivity. It is shown that Matlis reflexive
  modules form a Krull--Schmidt category. For noetherian rings the
  absence of infinite direct sums is a characteristic feature of
  Matlis reflexivity. This leads to a discussion of objects that are
  extensions of artinian by noetherian objects. 
  Classifications of Matlis reflexive modules are provided for some
  small examples.
\end{abstract}

\date{\today.}

\maketitle

\section{Introduction}

Matlis duality is one of the corner stones of commutative algebra. It
is also used in noncommutative algebra, because it provides for any
ring a relation between left and right modules. In this note we
revisit this classical theory which is based upon the seminal work of
Matlis \cite{Ma1958}. Of particular interest are Matlis reflexive
modules and we show that they form a Krull--Schmidt category. This
class of modules admits nice descriptions when the ring is noetherian
local and complete. For instance, they are precisely the modules that
are extensions of artinian by noetherian modules \cite{En1984,Zi1974,
  Zo1983}.  Another characteristic feature of Matlis reflexivity is
the absence of infinite direct sums. For any abelian category this
leads us to an analysis of objects that are extensions of artinian by
noetherian objects. The study of this finiteness condition goes back
to Baer who called abelian groups with this property \emph{minimax}
\cite{Ba1953}.

Let $A$ be a commutative ring and let $\Omega$ denote the set of
maximal ideals of $A$. We write
\[E\coloneqq E\left(\coprod_{\frm\in\Omega}A/\frm\right)\] for the
injective envelope of a coproduct of a representative set of simple
$A$-modules; it is a minimal injective cogenerator for the category of
$A$-modules. \emph{Matlis duality} for $A$-modules is given by the assignment
\[M\longmapsto M^\vee\coloneqq\Hom_A(M,E).\] For any pair of $A$-modules $M,N$
there is an isomorphism
\[\Hom_A(M,N^\vee)\cong \Hom_A(N,M^\vee)\]
and the evaluation map $M\to (M^\vee)^\vee$ is the unit of the adjunction.
The module $M$ is \emph{Matlis reflexive} if the evaluation map is an isomorphism. Examples
of Matlis reflexive modules are the modules of finite length.

\section{The Krull--Schmidt property}

Every finite length module admits an essentially unique decomposition
into indecomposables.  This generalises to Matlis reflexive modules.

\begin{thm}\label{th:KS}
  The Matlis reflexive modules form a Serre subcategory which is
  Krull--Schmidt. Thus for any exact sequence
  $0\to M'\to M\to M''\to 0$ the module $M$ is reflexive if and only
  $M'$ and $M''$ are reflexive. Moreover, any reflexive module admits
  an essentially unique decomposition into finitely many
  indecomposable modules with local endomorphism rings.
\end{thm}

The proof uses an intrinsic description of Matlis dual modules.

\begin{lem}\label{le:pure-inj}
An $A$-module is pure-injective if and only if it is a direct summand
of a Matlis dual module.
\end{lem}
\begin{proof}
  It its well known that a Matlis dual module is pure-injective; see
  \cite[Proposition~I.10.1]{Au1978} or
  \cite[Proposition~12.4.7]{Kr2022}. For the converse note that the
  canonical map $M\to (M^\vee)^\vee$ is a pure monomorphism, so it
  splits when $M$ is pure-injective.
\end{proof}  
  
\begin{proof}[Proof of Theorem~\ref{th:KS}]
  The first statement follows from an application of the Snake Lemma,
  since Matlis duality is an exact functor and the canonical map
  $M\to (M^\vee)^\vee$ is a monomorphism. For the second assertion we
  use the fact that a reflexive module is pure-injective, by
  Lemma~\ref{le:pure-inj}. A pure-injective module $M$ admits a
  decomposition $M=M'\oplus M''$ such that $M'$ is \emph{discrete}, so
  a pure-injective envelope of a direct sum $\bigoplus_{i\in I}M_i$ of
  indecomposable direct summands of $M$, and $M''$ is
  \emph{continuous}, so having no indecomposable direct summand
  \cite[Theorem~8.25]{JL1989}. From Lemma~\ref{le:Matlis} below it
  follows that $M''=0$ and that the set $I$ is finite, so
  $M=\bigoplus_{i\in I}M_i$. This yields the Krull--Schmidt property
  since each indecomposable pure-injective module has a local
  endomorphism ring.
\end{proof}

\begin{lem}
\label{le:Matlis}
Let $M=\bigoplus_{i\in I}M_i$ be a Matlis reflexive module. Then
$M_i=0$ for almost all $i$.  Moreover, any non-zero Matlis reflexive
module has an indecomposable direct summand.
\end{lem}
\begin{proof}
  We have an isomorphism $M^\vee\cong\prod_{i\in I}M_i^\vee$, and
  applying the duality to the canonical inclusion
    \[
    \bigoplus_{i\in I}M_i^\vee\lto \prod_{i\in I}M_i^\vee
    \]
    yields an epimorphism
    \[
      (M^\vee)^\vee\cong \left(\prod_{i\in I}M_i^\vee\right)^\vee\lto
      \left(\bigoplus_{i\in I}M_i^\vee\right)^\vee\cong \prod_{i\in
        I}(M_i^{\vee})^\vee.
    \]
    If $M$ is reflexive, this identifies with the canonical inclusion
    $M\to \prod_{i\in I}M_i$, so it is an isomorphism. Thus $M_i=0$
    for almost all $i$.

Suppose $M\neq 0$ has no indecomposable direct summand. Then we get a
decomposition $M=M_1\oplus M^1$ such that both $M_1$ and $M^1$ are
non-zero and admit no indecomposable direct summands. We continue and
decompose $M^i=M_{i+1}\oplus M^{i+1}$ as before, for all $i\ge
1$. This yields an inclusion $\bigoplus_{i\ge 1}M_i\to M$. If $M$ is
reflexive, then the same holds for $\bigoplus_{i\ge 1}M_i$. But this
is impossible by the first assertion. Thus $M$ admits an
indecomposable direct summand.
\end{proof}

\begin{rem}
 The isomorphism classes of reflexive modules form a set, because
  the isomorphism classes of indecomposable pure-injective modules
  form a set.
\end{rem}

\begin{rem}
  Morita duality provides another context where the notion of a reflexive
  module has been studied before; see \cite{Mo1958,Os1966}. This
  context is one way more general (as two rings are involved, not
  necessarily commutative), but also more restrictive (as the rings are
  necessarily semiperfect). Nonetheless, there are parallels. The first
  part of Lemma~\ref{le:Matlis} is the assertion of
  \cite[Lemma~2.13]{Os1966}, though the proofs are different.
\end{rem}  

 Let $\La$ be an associative $A$-algebra. Then Matlis duality
  $\Hom_A(-,E)$ induces an adjoint pair of functors between the
  categories of left and right $\La$-modules, respectively. The above
  Theorem~\ref{th:KS} remains true for $\La$-modules, with same
  proof. For an intrinsic description of Matlis reflexive $\La$-modules, see
  \cite{He1993,Kr1997}.

\begin{cor}
  For any associative algebra over a commutative ring, the Matlis
  reflexive modules form a Serre subcategory which is
  Krull--Schmidt.\qed
\end{cor}

Recall from \cite{Ze1953} that a module $M$ is \emph{linearly compact}
if for any codirected family of submodules $(M_i)_{i\in I}$ of $M$ and
any family of cosets $(x_i+M_i)_{i\in I}$ we have
$\bigcap_{i\in I} x_i+M_i\neq\varnothing$ provided that
$\bigcap_{i\in J} x_i+M_i\neq\varnothing$ for all finite
$J\subseteq I$.  An equivalent condition is that the canonical map
$M\to\varprojlim M_i$ is an epimorphism. This condition is the
categorical dual of the fact that for any directed family of
submodules $M_i\subseteq M$ the canonical map $\varinjlim M_i\to M$ is
a monomorphism.  Thus any Matlis reflexive module is linearly compact
(cf.\ Proposition~\ref{pr:reflexive}).
  
  The linearly compact modules form a Serre subcategory (cf.\
  Proposition~2, Proposition~3, and Proposition~9 of \cite{Ze1953}),
  and any decomposition $M=\bigoplus_{i\in I}M_i$ of a linearly
  compact module implies $M_i=0$ for almost all $i$ (cf.\ 
  Proposition~\ref{pr:reflexive}).  Moreover, if $M$ is linearly
  compact when viewed as a module over its endomorphism ring, then $M$
  is pure-injective \cite[Corollary~7.4]{JL1989}. Thus any linearly
  compact module over a commutative ring is pure-injective.  Combining
  these facts, it follows that the assertion of Theorem~\ref{th:KS}
  remains true for the class of linearly compact modules.

\begin{cor}
For any commutative ring, the linearly compact modules form a Serre subcategory which is
Krull--Schmidt.\qed
\end{cor}  

Note that commutativity is needed. For instance,
  any artinian module is linearly compact, but  the Krull--Schmidt property fails for artinian modules
  \cite{KR2000}.

\section{Matlis duality}

Now assume that the commutative ring $A$ is noetherian. Then we have the isomorphism
\[\coprod_{\frm\in\Omega}E(A/\frm)\cong
  E\left(\coprod_{\frm\in\Omega}A/\frm\right)\] since any coproduct of
injective modules is again injective. For any ideal $\fra$ we write
$A_\fra^\wedge=\varprojlim A/\fra^i$ for the completion with respect
to $\fra$.  We set
\[\widehat A\coloneqq \prod_{\frm\in\Omega}A_\frm^\wedge.\] 

\begin{lem}\label{le:End}
  There is an isomorphism of rings $ \End_A(E)\cong\widehat A$.
\end{lem}
\begin{proof}
For
$\frm\in\Omega$ we have
\[E(A/\frm)\cong\varinjlim\Hom_A(A/\frm^i,E(A/\frm))=\varinjlim\Hom_A(A/\frm^i,E)\]
since $E(A/\frm)$ is $\frm$-torsion. Thus
\[
  \begin{aligned}
  \End_A(E)&\cong \prod_{\frm\in\Omega}\Hom_A(E(A/\frm),E)\\
           &\cong \prod_{\frm\in\Omega}\Hom_A(\varinjlim\Hom_A(A/\frm^i,E),E)\\
           &\cong \prod_{\frm\in\Omega}\varprojlim\Hom_A(\Hom_A(A/\frm^i,E),E)\\
           &\cong \prod_{\frm\in\Omega}\varprojlim A/\frm^i\\
 &=\widehat  A.
\end{aligned}  
\]
The first three isomorphisms are clear since $\Hom_A(-,E)$ sends
colimits to limits. The last isomorphism uses that finite length
modules are Matlis reflexive.
\end{proof}

For an $\widehat A$-module $M$ and $\frm\in\Omega$ we set
$M(\frm)\coloneqq M A_\frm^\wedge$. We say that $M$ has \emph{finite
  support} if $M(\frm)=0$ for almost all $\frm\in\Omega$, and $M$ is
\emph{locally finitely supported} if the following equivalent
conditions hold:
\begin{enumerate}
\item[(LF1)] The canonical map
  $\bigoplus_{\frm\in\Omega}M(\frm)\to M$ is an isomorphism.
\item[(LF2)] $M$ is a filtered colimit of noetherian modules with
  finite support.
\end{enumerate}
From now on we restrict ourselves to the study of locally finitely
supported $\widehat A$-modules.  These modules form a locally
noetherian Grothendieck category, since
$\{A_\frm^\wedge\mid\frm\in\Omega\}$ is a set of noetherian
generators.

The following is essentially due to Matlis. In \cite{Ma1958} he
considered the case that $A$ is local. We write $\art A$ for the category of artinian $A$-modules and
$\noeth A$ for the category of noetherian $A$-modules. 

\begin{prop}\label{pr:equivalence}
  The  functors 
\[
  \begin{tikzcd}[column sep = huge]
  \art A=\art\widehat A
  \arrow[yshift=.75ex]{rrr}{\Hom_A(-,E)=\Hom_{\widehat A}(-,E)}
  &&&\noeth\widehat A
  \arrow[yshift=-.75ex]{lll}{\Hom_{\widehat A}(-,E)}
  \end{tikzcd}
\]
are mutually quasi-inverse contravariant equivalences.
\end{prop}
\begin{proof}
  Locally finitely supported modules that are artinian or noetherian
  have finite support. Thus we may restrict to modules that are
  supported on a finite subset $\Psi\subseteq\Omega$.  Let $A_\Psi$
  denote the localisation of $A$ with respect to
  $\bigcap_{\frm\in\Psi}(A\smallsetminus\frm)$. This ring is semilocal
  and its maximal ideals identify with the ones in $\Psi$.  Thus the
  completion of $A_\Psi$ with respect to the Jacobson radical
  $J(A_\Psi)$ is isomorphic to $\prod_{\frm\in\Psi}
  A_{\frm}^\wedge$. We conclude that the assertion reduces to the case
  that $A$ is semilocal. Then the $A$-module $E$ is artinian
  \cite[Lemma~2.4.19]{Kr2022} and therefore $\widehat A\cong\End_A(E)$
  is noetherian.

  Now suppose that $A$ is semilocal. Then the artinian $A$-modules are
  precisely the kernels of maps $E^p\to E^q$ for positive integers
  $p,q$, while the noetherian $\widehat A$-modules are presicely the
  cokernels of maps $\widehat A^p\to \widehat A^q$ for positive integers
  $p,q$. It remains to observe that $\Hom_A(-,E)$ is an exact functor
  that induces a contravariant
  equivalence $\add E\iso\add \widehat A$, with quasi-inverse given by
  $\Hom_{\widehat A}(-,E)$. Here, $\add M$ denotes the full
  subcategory of all direct summands of finite direct sums of copies
  of $M$, and the equivalence is clear from the isomorphism
  \[\End_A(E)\iso \End_{\widehat A}(\widehat A)=\widehat A\]
from Lemma~\ref{le:End}  which is induced by $\Hom_A(-,E)$. 
\end{proof}

\section{Categories of extensions}

Let $\A$ be an abelian category. We fix a pair of Serre subcategories
$\C,\D$ of $\A$ and write $\C\ast\D$ for the full subcategory
consisting of
objects $X\in\A$ that  fit into an exact sequence
\[0\lto X'\lto X\lto X''\lto 0\qquad\text{with} \qquad
  X'\in\C,\,X''\in\D.\]

We say that $\C\subseteq\A$ is \emph{right cofinal} if for every
epimorphism $X\to Y$ in $\A$ with $Y\in\C$ there exists an epimorphism
$\bar X\to Y$ in $\C$ that factors through $X\to Y$.

\begin{lem}\label{le:ext-cat}
  The subcategory $\C\ast\D$ of $\A$ is closed under subobjects and
  quotient objects. In particular, $\C\ast\D$ is an abelian
  category. Moreover, $\C\ast\D$ is closed under extensions provided
  that $\C\subseteq\A$ is right cofinal.
\end{lem}
\begin{proof}
The first assertion is clear. Now let $0\to X\to Y\to Z\to 0$ be an
exact sequence in $\A$ with subobjects $X'\subseteq X$ and $Z'\subseteq Z$ in $\C$
such that $X''=X/X'$ and $Z''=Z/Z'$ are in $\D$. We consider the pullback
$Y\times_Z Z'$ and get an epimorphism $\pi\colon Y\times_Z Z'\to Z'$. The cofinality condition
yields an epimorphism $\bar Y\to Z'$ in $\C$ that factors through
$\pi$; in particular the composite
with $Z'\to Z$ factors through
$Y\to Z$ via a morphism $\bar Y\to Y$. We set
$Y'\coloneqq\Im(X'\oplus\bar Y\to Y)$ and $Y''\coloneqq Y/Y'$. Then we
have $Y'\in\C$ and an exact sequence $X''\to Y''\to Z''\to 0$ which
shows that $Y''\in\D$. Thus $Y$ lies in  $\C\ast\D$.
\end{proof}

\begin{lem}\label{le:arno}
The inclusions $\C\to\C\ast\D$ and  $\D\to\C\ast\D$ induce  equivalences
\[\frac{\C}{\C\cap\D}\longiso\frac{\C\ast\D}{\D}\qquad\text{and}\qquad
  \frac{\D}{\C\cap\D}\longiso\frac{\C\ast\D}{\C}.\]
\end{lem}
\begin{proof}
  We may assume $\A=\C\ast\D$.  The inclusion $\C\to\A$ is exact and
  the kernel of the composite with the quotient functor $\A\to\A/\D$
  equals $\C\cap\D$. It is clear that the composite is essentially
  surjective, because for any $X\in\A$ the monomorphism $X'\to X$ is
  an isomorphism in $\A/\D$. In order to show that the induced functor
$\C/(\C\cap\D)\to\A/\D$ is fully faithful we
  need to check the following cofinality condition: for each map
  $\p\colon X\to Y$ such that $Y\in\C$ and $\Ker\p$ and $\Coker\p$ are
  in $\D$, there is a morphism $\psi\colon X'\to X$ such that
  $X'\in\C$ and $\Ker(\p\psi)$ and $\Coker(\p\psi)$ are in $\D$. But
  this is clear, since there is such a monomorphism $\psi$ with
  $\Coker\psi\in\D$. The proof of the other equivalence is dual.
\end{proof}

Writing an object as an extension of objects in $\C$ and $\D$ is
unique up to objects from $\C\cap\D$.  To make this precise, fix a
pair of exact sequences $0\to X_i\to Y\to Z_i\to 0$ ($i=1,2$) with
$X_i\in\C$ and $Z_i\in\D$. We set $\bar X\coloneqq X_1\cap X_2$ and
this yields an exact sequence $0\to \bar X\to Y\to \bar Z\to 0$ such
that $\bar X\in\C$ and $\bar Z\in\D$. Moreover, the induced morphisms
$\bar X\to X_i$ and $\bar Z\to Z_i$ have kernels and cokernels in
$\C\cap\D$. In the language of quotient categories we see that $\C$
and $\D$ form a torsion pair, modulo the subcategory $\C\cap\D$.

\begin{lem}\label{le:torsion}
  The categories $\C$ and $\D$ yield a pair of full
  subcategories \[\left(\frac{\C}{\C\cap\D}, \frac{\D}{\C\cap\D}\right)\] of
  $(\C\ast\D)/(\C\cap\D)$ which form a torsion pair.
\end{lem}
\begin{proof}
  We may assume $\A=\C\ast\D$ and set $\bar\A=\A/(\C\cap\D)$.  The
  cofinality argument in the proof of Lemma~\ref{le:arno} shows that
  the inclusions $\C\to\A$ and $\D\to\A$ induce fully faithful
  functors $\bar\C=\C/(\C\cap\D)\to\bar\A$ and
  $\bar\D=\D/(\C\cap\D)\to\bar\A$, respectively.

  For a pair of full subcategories $(\T,\F)$ of an abelian category to
  be a torsion pair, it suffices that $\Hom(X,Y)=0$ for all $X\in\T$
  and $Y\in\F$, and that every object $X$ fits into an exact
  sequence $0\to X'\to X\to X''\to 0$ with $X'\in\T$ and $X''\in\F$.

  Any morphism in $\bar\A$ between objects in $\bar\C$ and $\bar\D$ is
  isomorphic to a morphism coming from $\A$, and we may assume it is a
  morphism between objects in $\C$ and $\D$.  The image of such a
  morphism is in $\C\cap\D$, so is zero in $\bar\A$.  For any $X\in\A$
  the sequence $0\to X'\to X\to X''\to 0$ with $X'\in\C$ and
  $X''\in\D$ induces the required exact sequence in $\bar\A$.
\end{proof}

\section{Artinian-by-noetherian objects}

Let $\A$ be an abelian category. We call an object $X$ in $\A$
\emph{artinian-by-noetherian} if it fits into an exact sequence
$0\to X'\to X\to X''\to 0$ such that $X'$ is noetherian and $X''$ is
artinian.

\begin{lem}\label{le:art-noeth}
  Let $\A$ be an abelian category and $0\to X\to Y\to Z\to 0$ an exact
  sequence such that $X$ is noetherian and $Z$ is artinian. Then $Y$
  has no subquotient that is an infinite coproduct of non-zero
  objects.
\end{lem}
\begin{proof}
  Any subquotient $Y'$ of $Y$ fits into an exact sequence
  $0\to X'\to Y'\to Z'\to 0$ such that $X'$ is noetherian and $Z'$ is
  artinian. Thus we may assume that $Y'=Y=\coprod_{i\in\bbN}Y_i$ and
  need to show that $Y_i=0$ for almost all $i$. For each $n\in\bbN$
  let $X_n$ denote the pull-back of $X\to Y\leftarrow \coprod_{i\le n}Y_i$
  and  $Z_n$  the push-out of $\coprod_{i\le n}Y_i\leftarrow Y\to Z$.
We obtain chains
\[X_0\rightarrowtail X_1\rightarrowtail X_2\rightarrowtail
  \cdots\rightarrowtail X\qquad\text{and}\qquad  Z\twoheadrightarrow\cdots\twoheadrightarrow
  Z_2\twoheadrightarrow Z_1\twoheadrightarrow Z_0
 \]
 which stabilise, say for $n\ge n_0$, so $X_{n_0}=X$ and $Z=Z_{n_0}$.
 This yields an induced exact sequence
 $0\to X\to \coprod_{i\le n_0}Y_i\to Z\to 0$ and therefore $Y_i=0$ for
 all $i> n_0$.
\end{proof}

The converse of the preceding lemma requires an assumption. Recall
that a Grothendieck category is \emph{locally noetherian} if there is
a generating set of noetherian objects.  This means in particular that
finitely generated and noetherian objects coincide. In addition we
assume that the injective envelope of every simple object is artinian.
Examples are the category of modules over a commutative noetherian
ring, or more generally the category of quasi-coherent sheaves on a
noetherian scheme. Also, modules over a noetherian algebra have this
property.

\begin{lem}
  Let $\La$ be an $A$-algebra such that $\La$ is noetherian as an
  $A$-module. Then the injective envelope of each simple $\La$-module is artinian.
\end{lem}
\begin{proof}
We may assume that the ring $A$ is noetherian and then $E(A/\frm)$ is artinian
for every $\frm\in\Omega$; see  \cite[Lemma~2.4.19]{Kr2022}. It
follows that $\Hom_A(\La,E(A/\frm))$ is an artinian $\La$-module for
each $\frm$. Let $S$ be a simple $\La$-module. The
$\La$-module
\[\Hom_A(\La,E)\cong\coprod_{\frm\in\Omega}\Hom_A(\La,E(A/\frm))\]
is an injective cogenerator. Thus for some $\frm\in\Omega$ the injective
envelope $E(S)$ is a direct summand of $\Hom_A(\La,E(A/\frm))$ and
therefore artinian.
\end{proof}

Here is the converse of Lemma~\ref{le:art-noeth}.

\begin{lem}\label{le:subquot}
  Let $\A$ be a locally noetherian Grothendieck category and
 suppose that injective envelopes of simple objects are artinian. Let
 $X$ be an
  object having no subquotient that is an infinite coproduct of
  non-zero objects. Then $X$ fits into an exact sequence
  $0\to X'\to X\to X''\to 0$ such that $X'$ is noetherian and $X''$ is
  artinian.
\end{lem}

\begin{proof}
  The assumption on $\A$ implies that an object is artinian if and
  only if its socle is an essential subobject and has finite length.

  We construct a noetherian subobject $X'\subseteq X$ as follows.
  Suppose that $X$ is not artinian. From our first observation it
  follows that  $\soc X\subseteq X$
  is not essential, so there is a finitely generated subobject
  $0\neq V\subseteq X$ with $V\cap\soc X=0$. Choose a maximal
  subobject $U\subseteq V$. Then the canonical map $X\to X/U$ induces
  a proper inclusion $\soc X\to \soc (X/U)$. If $X/U$ is not artinian,
  we proceed and obtain a sequence of epimorphisms
  $X=X_0\to X_1\to X_2\to\cdots$ such that $\Ker (X_i\to X_{i+1})$ is
  finitely generated and the induced map $\soc X_i\to\soc X_{i+1}$ is
  a proper inclusion for each $i\ge 0$. This must terminate after
  finitely many steps since
  $\varinjlim \soc X_i\subseteq\soc (\varinjlim X_i)$ has finite
  length, again by our assumption on $X$. Thus for some $i\ge 0 $ the
  quotient $X''=X_i$ is artinian, and $X'=\Ker (X\to X'')$ is finitely
  generated.
\end{proof}

\begin{prop}\label{pr:art-noeth}
  Let $\A$ be a locally noetherian Grothendieck category and suppose
  that injective envelopes of simple objects are artinian. Then the
  artinian-by-noetherian objects form a Serre subcategory in $\A$; it
  is the smallest Serre subcategory containing both, the artinian and
  the noetherian objects. Moroever, it consists of the objects having
  no subquotients that are infinite coproducts of non-zero objects.
\end{prop}
\begin{proof}
It follws from Lemma~\ref{le:ext-cat} that the  artinian-by-noetherian
objects form a Serre subcategory, because the noetherian objects are
right cofinal in $\A$. The other description of the   artinian-by-noetherian
objects follows from Lemmas~\ref{le:art-noeth} and \ref{le:subquot}.
\end{proof}

Writing an object as an extension of an artinian and a noetherian
object is unique up to objects of finite length. We make this more
precise, using the language of quotient categories.

We write $\arno\A$ for the full subcategory consisting of all
artinian-by-noetherian objects in $\A$.\footnote{This is an abuse of
  notation, as Auslander introduces in \cite{Au1978} the same notation
  for modules that are extensions of noetherian by artinian
  modules. For instance, almost split sequences over noetherian
  algebras provide such extensions.} As usual, $\art\A$, $\noeth\A$,
and $\fl\A$ denote the full subcategories of objects in $\A$ that are
artinian, noetherian, and finite length, respectively.

Let us record the basic properties of $\arno\A$; they are immediate
consequences of  Lemmas~\ref{le:arno} and \ref{le:torsion}.

\begin{prop}\label{pr:torsion}
  The categories $\noeth\A$ and $\art\A$ yield a pair of full subcategories
  \[\left(\frac{\noeth\A}{\fl\A},\frac{\art\A}{\fl\A}\right)\] of $(\arno\A)/(\fl\A)$ which form a
  torsion pair.\qed
\end{prop}

\begin{prop}\label{pr:arno}
  \pushQED{\qed}
  The inclusions $\art\A\to\arno\A$ and  $\noeth\A\to\arno\A$ induce equivalences
\[\frac{\art\A}{\fl\A}\longiso\frac{\arno\A}{\noeth\A}\qquad\text{and}\qquad
\frac{\noeth\A}{\fl\A}\longiso\frac{\arno\A}{\art\A}.\qedhere\]
\end{prop}

\section{Matlis reflexive modules}

We keep the assumption that the commutative ring $A$ is noetherian. It
is convenient to call an $A$-module
\emph{complete} if the $A$-action factors through the canonical map
$A\to \widehat A$. Clearly, any Matlis dual module is complete. For
that reason we restrict ourselves to the study of complete
$A$-modules, which identify with $\widehat A$-modules.  This leads to
a characterisation of Matlis reflexive modules as extensions of
artinian by noetherian modules, which is well known when the ring $A$
is local \cite{En1984,Zi1974, Zo1983}.

\begin{prop}\label{pr:reflexive}
  For a locally finitely supported $\widehat A$-module $M$ the following are equivalent.
  \begin{enumerate}
  \item $M$ is Matlis reflexive.
  \item $M$ is linearly compact.  
  \item $M$ has no subquotient that is an infinite
    direct sum of non-zero modules.
  \item $M$ has a noetherian submodule $U$ such that  $M/U$ is  artinian.
  \end{enumerate}
\end{prop}

\begin{proof}
  (1) $\Rightarrow$ (2)
Given a codirected family of submodules $M_i\subseteq M$, the canonical
map $\varinjlim (M/M_i)^\vee\to M^\vee$ is a momomorphism. Thus
\[M\cong (M^\vee)^\vee\lto (\varinjlim(M/M_i)^\vee)^\vee
  \cong\varprojlim ((M/M_i)^\vee)^\vee \cong\varprojlim M/M_i\]
is an epimorphism.

(2) $\Rightarrow$ (3) Any subquotient of a linearly compact module is
again linearly compact. Thus we may assume that there is a
decomposition $M=\bigoplus_{i\in I}M_i$. The
cofinite subsets $J\subseteq I$ yield a codirected family of
submodules $M_J=\bigoplus_{i\in J}M_i$ of $M$. Choose $x_i\in M_i$
for each $i\in I$ and set $x_J=\sum_{i\not\in J}x_i$. For any finite
set of cofinite subsets $J_\a\subseteq I$ we have $x_J\in\bigcap_\a
x_{J_\a}+M_{J_\a}$ for $J=\bigcap_\a J_\a$. Thus there exists  $x\in \bigcap_J
x_{J}+M_{J}$ when $J$ runs through all cofinite subsets. This satisfies $x+\bigoplus_{j\neq i}M_j= x_i+\bigoplus_{j\neq
  i}M_j$ for all $i\in I$, and therefore $x_i=0$ for almost all $i$.

(3) $\Rightarrow$ (4) Suppose $M$ has no subquotient that is an
  infinite direct sum of non-zero modules. Then $M$ has finite support
  and  therefore it may be viewed as a module over a factor of $\widehat A$ which
  is noetherian.  Thus we can apply Lemma~\ref{le:subquot}.

  (4) $\Rightarrow$ (1) Artinian and noetherian $\widehat A$-modules
  are reflexive, by Proposition~\ref{pr:equivalence}. As reflexive
  modules are closed under extensions, by Theorem~\ref{th:KS}, it
  follows that $M$ is reflexive.
\end{proof}

We write $\refl\widehat A$ for the category of Matlis reflexive
$\widehat A$-modules
and $\fl\widehat A$ for the category of finite length $\widehat A$-modules. Then we
obtain from Proposition~\ref{pr:torsion} the torsion pair
\[\left(\frac{\noeth\widehat A}{\fl\widehat A},\frac{\art\widehat A}{\fl\widehat A}\right)\] in
$(\refl\widehat A)/(\fl\widehat A)$ and from Proposition~\ref{pr:arno} the
pair of equivalences
\[\frac{\art\widehat A}{\fl\widehat A}\longiso\frac{\refl\widehat
    A}{\noeth\widehat A}\qquad\text{and}\qquad \frac{\noeth\widehat
    A}{\fl\widehat A}\longiso\frac{\refl\widehat A}{\art\widehat A}.\]

\section{Examples}

We provide descriptions of Matlis reflexive modules for some small
examples. The first example illustrates why one passes
from a noetherian ring $A$ to the completion $\widehat A$. Because of the
Krull--Schmidt property, it suffices to list the indecomposbable
modules. Moreover, any indecomposable $\widehat A$-module is the
restriction of a module over the completion $A^\wedge_\frm$ for some
maximal ideal $\frm$ of $A$. Thus it suffices to consider rings that
are local and complete.

\begin{exm}
  Let $A= \bbZ$. It is easily checked that a $\bbZ$-module is Matlis reflexive if and only if it
  has finite length.
\end{exm}

\begin{exm}\label{ex:dedekind}
  Let $A$ be a Dedekind domain. Passing from $A$ to $\widehat A$, we may
  assume that $A$ is a complete discrete valuation ring, so local and
  complete of Krull dimension one. By inspection it is easily checked
  that an indecomposable $A$-module is Matlis reflexive if and only if
  it is pure-injective.  The following is the complete list of
  indecomposable Matlis reflexive modules, where $\frm$ denotes the
  maximal ideal:
  \[A=\varprojlim A/\frm^i\,,\quad E(A/\frm)=\varinjlim
    A/\frm^i\,,\quad Q(A)\,,\quad A/\frm^i\;\; (i\in\bbN).\] The
  category $\refl A=\arno A$ is an abelian category with enough
  projective and enough injective objects. Note that the quotient ring
  $Q=Q(A)$ is the unique indecomposable object that is projective and
  injective.  Set $\bar A= \smatrix{A&Q\\ 0&Q}$ and
  $T_2(Q)=\smatrix{Q&Q\\ 0&Q}$. Then $\Hom_A(A\oplus Q,-)$ induces an
  equivalence $\refl A\iso\mod\bar A$, and
  passing to the quotient modulo $\fl A$  yields equivalences
\[\frac{\refl A}{\fl A}\simeq\mod T_2(Q)\qquad
  \text{and}\qquad \frac{\noeth  A}{\fl A}\simeq \mod Q\simeq \frac{\art  A}{\fl A}.\]
\end{exm}

\begin{exm}
  Let $k$ be a field and $A=k[x,y]/(xy)$. There are two minimal prime
  ideals $\mathfrak x=(x)/(xy)$ and $\mathfrak y=(y)/(xy)$, with
  $A/\mathfrak x\cong k[y]$ and $A/\mathfrak y\cong k[x]$. For the
  maximal ideal $\frn=(x,y)/(xy)$ we have $A/\frn\cong k$.  Thus the
  ring $A$ fits into a pullback diagram:
  \[
    \begin{tikzcd}[column sep=small,row sep=small]
      &k[x]\arrow[rd]\\
      A\arrow[ru]\arrow[rd]&&k\\
      &k[y] \arrow[ru]
    \end{tikzcd}
  \]
  In particular, $\Spec A$ is given by two copies of an affine line,
  glued via the distinguished maximal ideal $\frn$:
  \[\Spec A=\Spec k[x]\amalg_{\frn}\Spec k[y]\] For the set
  of maximal ideals we have
  \[\Om_A= (\Om_{k[x]}\smallsetminus\{(x)\})\cup \{\frn\}\cup (\Om_{k[y]}\smallsetminus\{(y)\}).\]

  For $\frm\in\Om_A$ and $\frm\neq\frn$, the ring
  $A_\frm^\wedge\cong \End_A(E(A/\frm))$ is a complete discrete
  valuation ring, because $E(A/\frm)$ is annihilated either by
  $\mathfrak x$ or by $\mathfrak y$ and identifies therefore with an
  indecomposable injective module over a polynomial ring.  In
  particular, the description of the Matlis reflexive modules over
  $A_\frm^\wedge$ follows from the preceding example.

 For  $\frm=\frn$ we have $A^\wedge_\frm\cong k\llbracket
  x,y\rrbracket/(xy)$ and we refer to the next example for a
  description of its Matlis reflexive modules.

  The algebra $A$ is a string algebra. Crawley-Boevey offers in
  \cite{CB2018} a detailed discussion of modules over string algebras,
  including a classification of the artinian modules, and using the
  functorial filtration method, which goes back to Gelfand and
  Ponomarev \cite{GP1968}. Note that $\art A=\art\widehat A$; so one may compare our
  description with the one given in \cite[Theorem~1.3]{CB2018}.
\end{exm}

\begin{exm}
  Let $k$ be a field and $A=k\llbracket x,y\rrbracket/(xy)$. This ring is a
  completion of $k[x,y]/(xy)$ and we have therefore an analogous  pullback diagram:
  \[
    \begin{tikzcd}[column sep=small,row sep=small]
      &k\llbracket x\rrbracket\arrow[rd]\\
      A\arrow[ru]\arrow[rd]&&k\\
      &k\llbracket y\rrbracket \arrow[ru]
    \end{tikzcd}
  \]
  The following is the complete list of indecomposable Matlis
  reflexive $A$-modules:
\begin{enumerate}
 \item string modules $M(C)$ given by finite or stabilising strings $C$, 
 \item band modules $M(C,V)$ given by periodic strings $C$ and
   indecomposable $k[t,t^{-1}]$-modules $V$ of finite length.
 \end{enumerate}
 This list can be deduced from work of Ebrahimi~Atani \cite{At2000},
 though our notation and terminology follows more closely that of
 \cite{CB2018}. For string modules given by infinite strings, see also
 Ringel \cite{Ri1995}.  The isomorphism type of a string
 or band module is essentially determined by the corresponding string
 $C$ and the isomorphism type of the $k[t,t^{-1}]$-module $V$; see
 \cite{CB2018} for details.
 
A \emph{string} is a possibly infinite word in the alphabet
$\{x,y,x^-,y^-\}$ having no subword of the form $xy$, $yx$, $x^-y^-$,
$y^-x^-$. We identify strings with diagrams where the letters are
represented by arrows. Given a string $C$, we write $C^-$ for the
string which is obtained from $C$ by changing each letter $a$ in $C$
to $a^-$, where $(a^-)^-=a$.  A string is either finite or infinite,
and for infinite strings we distinguish between type $\bbN$ and type
$\bbZ$.  A string $C$ can be written as a sequence $(C_i)_{i\in I}$
indexed by an interval $I\subseteq\bbZ$.  We say that an infinite
string $C$ is \emph{stabilising} if $C_i=C_{i+1}$ for $|i|\gg 0$. We
say that $C$ is \emph{periodic} if there exists some $n>1$ such that
$C_i=C_{i+n}$ for all $i\in\bbZ$. The
minimal such $n$ is called the \emph{period}.

Let us define the \emph{string module} $M(C)$ over $A$ for a string
$C$ that is finite or stabilising. For finite $C$ the module $M(C)$ has
a $k$-basis that is indexed by the vertices of $C$ and the arrows show
the actions of $x$ and $y$.  For an infinite string  we use truncations
and note that they are again stabilising or finite. So for $C$ of the form
\[
  \begin{tikzcd}[column sep=15pt]
    \cdots \arrow[r,"a"]&\circ_2 \arrow[r,"a"] & \circ_1\arrow[r,"a"] &
    \circ_0\arrow[r,dash]&\circ\arrow[r,dash]&\cdots
 \end{tikzcd}
\]
(with $a=x$ or $a=y$) one takes the direct limit of the string modules given by the truncations
\[
  \begin{tikzcd}[column sep=15pt]
    \circ_i\arrow[r,"a"] &\cdots \arrow[r,"a"] & \circ_1\arrow[r,"a"] &
    \circ_0\arrow[r,dash]&\circ\arrow[r,dash]&\cdots
 \end{tikzcd}
\]
and for a string of the form
\[
  \begin{tikzcd}[column sep=15pt]
    \cdots&\circ_2\arrow[l,swap,"a"] & \circ_1\arrow[l,swap,"a"] &
    \circ_0\arrow[l,swap,"a"] \arrow[r,dash]&\circ\arrow[r,dash]&\cdots
   \end{tikzcd}
 \]  
 one takes  the inverse limit of the string modules given by the truncations
\[
  \begin{tikzcd}[column sep=15pt]
   \circ_i& \cdots\arrow[l,swap,"a"] & \circ_1\arrow[l,swap,"a"] &
    \circ_0\arrow[l,swap,"a"] \arrow[r,dash]&\circ\arrow[r,dash]&\cdots
   \end{tikzcd}
 \]
 The definition of $M(C)$ depends on choices, but it can be shown that the isomorphism type of the
  $A$-module $M(C)$ is independent of any choices. For this one uses
  the observation that any
 concatenation $C=C'C''$ via an arrow from $C''$ to $C'$ yields an
 exact sequence
 \[0\lto M(C')\lto M(C)\lto M(C'')\lto 0, \] and there is an obvious
 swap when the connecting arrow goes in the other direction.

There are two distinguished strings:
\[
  \begin{tikzcd}[column sep=15pt]
    C_\infty\colon&\cdots \arrow[r,"x"]&\circ \arrow[r,"x"]
     & \circ \arrow[r,"x"] & \circ &
      \circ \arrow[l,swap,"y"] &\circ
      \arrow[l,swap,"y"]&\cdots \arrow[l,swap,"y"] 
    \end{tikzcd}
  \]
  and  
 \[ \begin{tikzcd}[column sep=15pt]
      D_\infty\colon&\cdots&\circ\arrow[l,swap,"x"] & \circ
      \arrow[l,swap,"x"] & \circ\arrow[l,swap,"x"] \arrow[r,"y"]&
      \circ \arrow[r,"y"] & \circ\arrow[r,"y"] &\cdots
    \end{tikzcd}
\]  
Let $\frm$ denote the maximal ideal of $A$. We claim that
$M(C_\infty)\cong E(A/\frm)$. In fact, we have for $i\ge 1$
\[\soc^i E(A/\frm)=\Hom_A(A/\frm^i, E(A/\frm)\cong M(C_{i-1})\]
with
  \[ \begin{tikzcd}[column sep=15pt] C_i\colon&\circ_i \arrow[r,"x"]
      &\cdots \arrow[r,"x"] & \circ_1 \arrow[r,"x"] & \circ_0 &
      \circ_1 \arrow[l,swap,"y"] & \cdots \arrow[l,swap,"y"] &\circ_i
      \arrow[l,swap,"y"]
    \end{tikzcd} \]
  and therefore
  \[M(C_\infty)=\varinjlim M(C_i) \cong \varinjlim\soc^i
    E(A/\frm)=E(A/\frm).\] For a finite string $C$ we have
  $M(C)^\vee\cong M(C^-)$, and therefore
\[M(D_\infty)=\varprojlim M(C^-_i)\cong
  E(A/\frm)^\vee=\Hom_A(E(A/\frm),E(A/\frm))\cong A.\]

For any finite or stabilising string $C$, the $A$-module $M(C)$ is an
extensions of an artinian by a noetherian module; see
Lemma~\ref{le:arno-ex} below.  In particular, $M(C)$ is
Matlis reflexive.  Note that $M(C)$ is artinian if and only if $C$
contains no infinite substring of the form
  \[ \begin{tikzcd}[column sep=15pt] \cdots&\circ\arrow[l,swap,"x"] &
      \circ\arrow[l,swap,"x"] & \circ\arrow[l,swap,"x"]
   \end{tikzcd}
   \quad\text{or}\quad
 \begin{tikzcd}[column sep=15pt] 
   \circ\arrow[r,"y"]& \circ\arrow[r,"y"] & \circ\arrow[r,"y"] &\cdots
    \end{tikzcd}
  \]
Analogously,  $M(C)$ is noetherian if and only if $C$ contains no
infinite substring of the form
\[ \begin{tikzcd}[column sep=15pt] \cdots
    \arrow[r,"x"]&\circ\arrow[r,"x"] & \circ\arrow[r,"x"] & \circ
 \end{tikzcd}
   \quad\text{or}\quad
   \begin{tikzcd}[column sep=15pt] \circ& \circ\arrow[l,swap,"y"]
     &\circ\arrow[l,swap,"y"]&\cdots \arrow[l,swap,"y"]
   \end{tikzcd} \]

 Now consider an infinite string $C$ of period $n$ and set
 $R=k[t,t^{-1}]$.  The string module $M(C)$ has a $k$-basis
 $\{b_i\mid i\in\bbZ\}$ that is indexed by the vertices of the diagram
 representing $C$, and the action of $x$ and $y$ is given by the
 arrows of this diagram. Then $M(C)$ is a free $R$-module of rank $n$
 with the action of $t$ given by $tb_i =b_{i+n}$. For any $R$-module
 $V$ one defines the \emph{band module} $M(C,V)=V\otimes_R M(C)$.

For the quotient category $(\refl A)/(\fl A)$ one observes that each
indecomposable object is either annihilated by $x$ or by $y$. Thus the
description reduces to modules over $k\llbracket x\rrbracket$ or
$k\llbracket y\rrbracket$, and we get an equivalence
 \[\frac{\refl A}{\fl A}\simeq\mod T_2(k(\!(x)\!))\times \mod T_2(k(\!(y)\!)).\]
\end{exm}

The following lemma has been used in the preceding example.

\begin{lem}\label{le:arno-ex}  
 Let $A=k\llbracket x,y\rrbracket/(xy)$. For a finite or stabilising string $C$
 the  module $M(C)$ is artinian-by-noetherian over $A$. Moreover,  $M(C)^\vee\cong M(C^-)$.
\end{lem}

\begin{proof}
  Clearly, $M(C)$ is of finite length if and only if $C$ is
  finite.

Next consider the following type $\bbN$ strings, where $a=x$ or
$a=y$:
\[ \begin{tikzcd}[column sep=15pt]
     C_a\colon\;\;\circ&\circ\arrow[l,swap,"a"]& \circ\arrow[l,swap,"a"] & \arrow[l,swap,"a"] \cdots
   \end{tikzcd}
   \quad\text{and}\quad
 \begin{tikzcd}[column sep=15pt] 
   D_a\colon\;\;\circ\arrow[r,"a"]& \circ\arrow[r,"a"] & \circ\arrow[r,"a"] &\cdots
    \end{tikzcd}
  \]
The module $M(C_a)$ equals the image of
  $M(C_\infty)\xto{a}M(C_\infty)$. Analogously, $M(D_a)$ equals the
  image of $M(D_\infty)\xto{a}M(D_\infty)$. Thus $M(C_a)$ is artinian
  and $M(D_a)$ is noetherian.

  If $C$ is of type $\bbN$, then $C$ can be written as
  concatenation of two strings $C'$ and $C''$ such that
  $C'\in\{C_x,C_y,D_x,D_y\}$ and $C''$ is finite.
  This gives rise to an extension of the form
  \[0\lto M(C')\lto M(C)\lto M(C'')\lto 0 \] or 
  \[0\lto M(C'')\lto M(C)\lto M(C')\lto 0.\] It follows that $M(C)$ is
  artinian or noetherian, since $M(C')$ is artinian or
  noetherian by our previous observation.

  If $C$ is of type $\bbZ$, then one can write it as concatenation of
  two type $\bbN$ strings $C'$ and $C''$, connected by an arrow from
  $C''$ to $C'$. This yields an exact sequence
  \[0\lto M(C')\lto M(C)\lto M(C'')\lto 0 \] and therefore $M(C)$ is
  artinian-by-noetherian; see Proposition~\ref{pr:art-noeth}.

  With the arguments from the first part of the proof, the assertion
  about Matlis duality reduces to the case that $C$ is finite or of
  the form $C_a$ or $D_a$. Here, one uses the exactness of Matlis
  duality and that $(\varinjlim M_i)^\vee\cong \varprojlim M_i^\vee$.
\end{proof}

\subsection*{Acknowledgements} 
I am grateful to Bill Crawley-Boevey for much help, in particular with
the Krull--Schmidt property for Matlis reflexive modules. Also, I wish
to thank Wassilij Gnedin, Norihiro Hanihara, and the referees for
helpful comments on this work. Alberto Fernandez Boix kindly pointed
me to \cite{Zi1974}. This work was supported by the Deutsche
Forschungsgemeinschaft (SFB-TRR 358/1 2023 - 491392403).

\end{document}